\documentclass[11pt]{article}

\usepackage{amsmath,amssymb,amsthm}
\usepackage[utf8]{inputenc}

\usepackage{tikz}

\newcommand\NN{\mathbb{N}}

\newcommand\RR{\mathbb{R}}

\theoremstyle{plain}
\newtheorem{theorem}{Theorem}[section]
\newtheorem{proposition}{Proposition}[section]

\newtheorem{lemma}{Lemma}[section]
\newtheorem{corollary}{Corollary}[section]

\theoremstyle{definition}
\newtheorem{definition}{Definition}[section]
\newtheorem{remark}{Remark}[section]
\newtheorem{example}{Example}[section]
\newtheorem{question}{Question}[section]
\newtheorem{claim}{Claim}

\newcommand{\CC}{\mathbb{C}} %% Conjunto complejos:     \C
\newcommand{\e}{\varepsilon}
\def\la{\lambda}

\newcommand{\card}{\operatorname{card}}

\def\e{\varepsilon}

\title{Ces\`{a}ro bounded  operators in Banach spaces}
\author{T. Berm\'{u}dez,
A. Bonilla,
V. M\"{u}ller\;
and A. Peris
\thanks{The first, second and four author were supported in part by MEC and FEDER,
Project MTM2016-75963-P. The third author was supported by grant No. 17-27844S of GA CR and RVO: 67985840.}}

\begin{document}

\maketitle
\begin{abstract}
We study several notions of boundedness for operators. It is known that any
 power bounded operator is absolutely Cesàro bounded and strong Kreiss bounded (in particular, uniformly Kreiss bounded). The converses do not hold in general.
In this note, we give examples of topologically mixing absolutely Ces\`{a}ro bounded operators on $\ell^p(\mathbb{N})$, $1\le p < \infty$, which are not
power bounded, and provide examples of uniformly Kreiss bounded operators which are not absolutely Ces\`{a}ro bounded. These results complement very limited number of known examples  (see \cite{Shi} and \cite{AS}).
In \cite{AS} Aleman and Suciu   ask if every uniformly Kreiss bounded operator $T$ on a Banach spaces satisfies that $\lim_n\| \frac{T^n}{n}\|=0$. We solve this question for Hilbert space operators and, moreover,
we prove that, if $T$ is  absolutely Ces\`{a}ro bounded on a Banach (Hilbert) space, then $\| T^n\|=o(n)$  ($\| T^n\|=o(n^{\frac{1}{2}})$, respectively).
As a consequence, every absolutely Ces\`{a}ro bounded operator on a reflexive Banach space is mean ergodic, and  there exist mixing  mean ergodic operators on $\ell^p(\mathbb{N})$, $1< p <\infty$.
Finally, we give new examples of  weakly ergodic 3-isometries and study numerically hypercyclic $m$-isometries on finite or infinite dimensional Hilbert spaces. In particular,  all weakly ergodic  strict  3-isometries  on a Hilbert space  are  weakly numerically hypercyclic. Adjoints of unilateral forward weighted shifts which are strict $m$-isometries on  $\ell ^2(\NN)$ are shown to be hypercyclic.
\end{abstract}

\section{Introduction}

Throughout this article $X$ stands for a Banach space, the symbol $B(X)$  denotes the space of bounded linear  operators defined on $X$, and $X^*$ is the space of continuous linear functionals on $X$.

Given $T\in B(X)$, we  denote the Cesàro mean by
$$
M_n(T)x:=\frac{1}{n+1}\sum_{k=0}^n T^kx
$$
for all $x\in X$.

We need to recall  some definitions concerning the behaviour of the sequence of  Cesàro means $(M_n(T))_{n\in \NN}$.

\begin{definition}
A linear operator $T$ on a Banach space $X$ is called

\begin{enumerate}
\item \emph{Uniformly ergodic}  if $M_n(T)$  converges   uniformly.
\item \emph{Mean ergodic} if $M_n(T)$ converges in the strong topology of $X$.
\item \emph{Weakly ergodic} if $M_n(T)$ converges  in the weak topology of  $X$.
\item \emph{Absolutely Ces\`aro bounded} if there exists a constant $C > 0$ such that
$$
\sup_{N \in \mathbb{N}} \frac{1}{N} \sum_{j=1}^N \|T^j x\| \leq C \|x\|\;,
$$
for all $x\in X$.
\item \emph{Ces\`{a}ro bounded} if the sequence $(M_n(T))_{n\in \NN}$ is bounded.
\end{enumerate}
\end{definition}

An operator $T$ is said \emph{power bounded} if    there is a $C>0$ such that $\|T^n\| <C$  for all  $ n$.

\ \par

%The following example in $\Bbb R^2$ or $\Bbb C^2$ due  Assani by
%\[
%T:= \left(
%\begin{array}{cc}
% -1& 2  \\
%0 & -1  \\
%\end{array}
%\right)
%\]
% is a Ces\`{a}ro bounded but not power bounded because
% \[
%T^n= \left(
%\begin{array}{cc}
% (-1)^n& (-1)^{n-1}2n  \\
%0 & (-1)^n  \\
%\end{array}
%\right) \;.
%\]

The class  of absolutely Cesàro bounded operators was introduced by Hou and Luo in \cite{HL}.

\begin{definition}
An operator $T$ is said
\begin{enumerate}
\item \emph{Uniformly  Kreiss bounded}  if  there is a $C>0$ such that
$$
\left\|\sum_{k=0} ^{n} \lambda^{-k-1} T^k\right\| \le \frac{ C}{|\lambda|-1} \;\; \mbox { for all } |\lambda |>1 \mbox{ and } n=0,1,2, \cdots
$$
\item \emph{Strongly  Kreiss bounded}  if there is a $C>0$ such that
$$
\|(\lambda I-T)^{-k}\| \le \frac{ C}{(|\lambda|-1)^k} \;\; \mbox { for all } |\lambda |>1  \mbox{ and } k=1, 2, \cdots
$$
\item  Kreiss bounded  if there is a $C>0$ such that
$$
\|(\lambda I-T)^{-1}\| \le \frac{ C}{|\lambda|-1} \;\; \mbox { for all } |\lambda |>1.
$$
\end{enumerate}
\end{definition}

\begin{remark}
{\rm
\begin{enumerate}
\item In \cite{MSZ}, it  is proved that an operator $T$ is  uniformly Kreiss bounded  if and only if there is a $C$ such that
$$
\|M_{n}(\lambda T)\| \le C \;\; \mbox{ for } |\lambda |=1  \mbox{ and } n=0,1,2, \cdots.
$$
%In finite dimensional spaces, also  uniformly Kreiss bounded is equivalent to power bounded.

%\item In finite dimensional spaces, also  strong Kreiss bounded is equivalent to power bounded.

\item We recall that  $T$  is strongly Kreiss bounded if and only if
$$
\|e^{zT}\| \le M e^{|z|},  \mbox{ for all } z \in \mathbb{C}.
$$

\item In \cite{GZ08}, it is shown that every strong  Kreiss bounded operator is   uniformly Kreiss bounded. MacCarthy (see  \cite {Shi}) proved
that if  $T$ is  strong  Kreiss bounded
then $\|T^n\|\le Cn^{\frac{1}{2}}$.
\item There exist Kreiss bounded operators which are not Ces\`{a}ro bounded, and conversely \cite{SZ}.
\item On finite-dimensional Hilbert spaces, the classes of uniformly Kreiss bounded, strong Kreiss bounded,  Kreiss bounded and power bounded operators are equal.
\item Any absolutely Ces\`{a}ro bounded operator is uniformly Kreiss bounded.

\end{enumerate}
}
\end{remark}

Let $X$  be the space of all bounded analytic functions $f$   on the unit disk of the complex plane such that their derivatives $f'$ belong to the Hardy space $H^1$, endowed with the norm
$$
\|f\| = \|f\|_{\infty} + \|f\|_{H^1}\;.
$$
Then  the multiplication operator, $M_z$,  acting on $X$ is Kreiss bounded but it fails to be power bounded.
Moreover, this operator is not uniformly  Kreiss bounded (see \cite{SW}).

Furthermore, for the Volterra operator $V$ acting on $L^p[0,1]$, $1\le p\le \infty$, we have that $I-V$
is uniformly Kreiss bounded, for $p=2$ it is power bounded (see  \cite{MSZ}),  and it is asked if every uniformly
Kreiss bounded operator on a Hilbert space is  power bounded. This is related to the following question in \cite[page 279]{AS} (see also, \cite{Su}):

\begin{question}\label{pregunta1}
If $T$ is a uniformly Kreiss bounded operator on a Banach space, does it follow that $\lim_n\| \frac{T^n}{n}\|=0$?
\end{question}

Graphically, we show the implications between the above definitions.

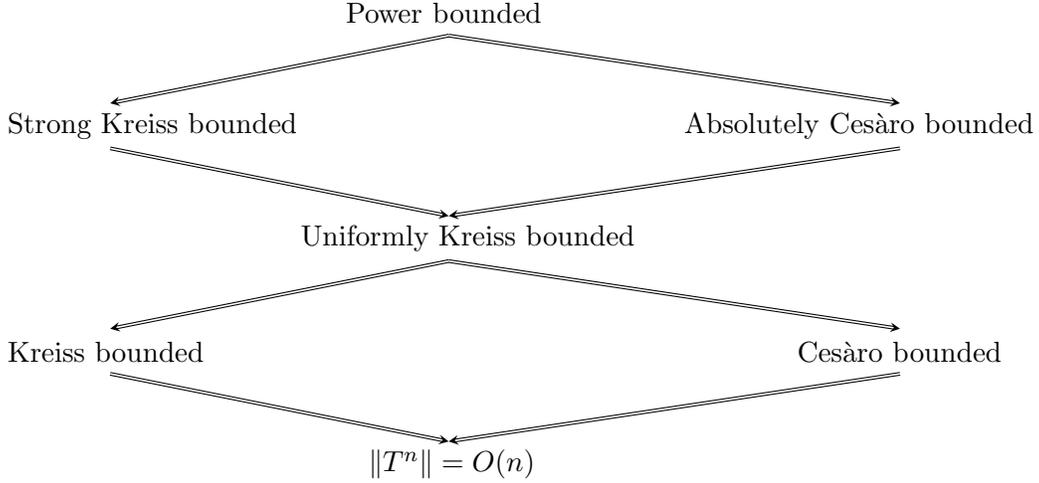
\begin{figure}[h]
\centering
%\hspace*{1cm}
\begin{tikzpicture}[scale=0.3,>=stealth]
  \node[right] at (25,15) {Power bounded};
  \node[right] at (10,10) {Strong Kreiss bounded};
  \node[right] at (40,10) {Absolutely Ces\`{a}ro bounded};
  \node[right] at (23,5) {Uniformly Kreiss bounded};
  \node[right] at (10,0) {Kreiss bounded};
  \node[right] at (45,0) {Ces\`{a}ro bounded};
  \node[right] at (26,-5) {\mbox {$\|T^n\|= O(n)$}};
  \draw[double, ->] (30,14) -- (15,11);
  \draw[double, ->] (30,14) -- (50,11);
  \draw[double, ->] (15,9) -- (30,6);
  \draw[double, ->] (50,9) -- (30,6);
   \draw[double, ->] (30,4) -- (15,1);
    \draw[double, ->] (30,4) -- (50,1);
    \draw[double, ->] (15,-1) -- (30,-4);
  \draw[double, ->] (50,-1) -- (30,-4);
 \end{tikzpicture}
\caption{Implications among different definitions related with
Kreiss bounded  and  Cesàro bounded operators in Banach spaces. }
\end{figure}

We recall some definitions that allow us to study some properties  of orbits  related to  the behavior of the sequence $(M_n(T))_{n\in \NN}$.
\begin{definition} Let $T\in B(X)$.
$T$ is \emph{topologically mixing} if for any pair $U,V$ of non-empty
open subsets of $X$, there exists some $n_0 \in \NN$ such that
$T^n(U) \cap V \neq \emptyset$ for all $n \geq n_0$.

%\item  $x$ has a  \emph{distributionally unbounded orbit} if there exist
% $B \subset \Bbb N$ with $\udens(B) = 1$ such that
%$\displaystyle \lim_{n \in B} \|T^nx\| = \infty$.
%\end{enumerate}
\end{definition}

%In \cite{BBPW}, it  is proved that, if $T$ is an absolutely Ces\`{a}ro bounded operator, then $T$ has no   distributionally
%unbounded orbits. As consequence,  frequently hypercyclic operators and  distributionally chaotic operators
%are not absolutely Ces\`{a}ro bounded. See \cite{BBPW} for more details.

Examples of   absolutely Ces\`{a}ro bounded
mixing operators on $\ell^p(\mathbb{N})$ are given in \cite{MOP13} (see Section 3.7 in \cite{beltran14}), \cite{HL}, and \cite{BBMP} (see \cite{BBPW}).

 \bigskip

Let $H$ be a Hilbert space. For a positive integer $m$, an operator $T\in B(H)$ is called an \emph{$m$-isometry} if for any $x\in H$,
$$
\sum _{k=0}^m (-1)^{m-k}{ m\choose k}\|T^{k}x\|^2 =0\; .
$$

We say that $T$ is a \emph{strict $m$-isometry } if $T$ is an $m$-isometry but it is not an $(m-1)$-isometry.

\ \par

\begin{remark}
{\rm  \begin{enumerate}
\item For $m\ge 2$, the strict $m$-isometries are not power bounded. Moreover,  $\|T^n\| =O(n)$ for  $3$-isometries and
$\|T^n\| =O(n^{\frac{1}{2}})$ for  $2$-isometries.
\item There are no strict $m$-isometries on finite dimensional spaces for $m$ even. See \cite[Proposition 1.23]{AgS}.
\item An example of weak ergodic $3$-isometry is provided in \cite{AS}.
\end{enumerate}
}
\end{remark}

%In this paper, we give a new examples of weak ergodic $3$-isometries.

%For $T\in B(X)$ and $(x,f) \in X\times X^*$, the \emph{numerical orbit} of $(x,f)$ is defined  as  the set
%$$
%O(T,x,f) = \{f(T^nx): n\in \Bbb Z^+\} \;.
%$$

%\begin{definition}
%Let
%$$
%\prod  (X) := \{ (x,f)\in X\times X^*: \|x\|=\|f\| =f(x) =1\}\;.
%$$
%It is said that the pair  $(x,f) \in \prod (X) $ is a \emph{numerically hypercyclic vector} for $T\in B(X)$ if  $O(T,x,f)$ is dense on $\Bbb C$.
%An operator $T$ is called \emph{numerically hypercyclic} if it has a numerically hypercyclic vector. We denote by  $NH(X)$ the set of all numerically hypercyclic operators $T\in B(X)$.
%\end{definition}

%Two operators $T\in B(X)$ and $S\in B(Y)$ are called \emph{similar} if there is an isomorphism $J:X\rightarrow Y$ such that $J^{-1}SJ =T$. We recall that numerical  hypercyclicity is  not an invariant by similarities in general \cite{S}.

%\begin{definition} Let $T\in B(X)$.  We say that $T$ is \emph{weakly numerically hypercyclic} if $T$ is similar to a numerically hypercyclic operator.  We denote the set of weakly numerically hypercyclic operators by $WNH(X)$.
%\end{definition}

%In \cite[Proposition 1.5]{S},  Shkarin  proved that  $T$ is weakly numerically hypercyclic if and only if there exist $x\in X$ and $f\in X^*$ such that $O(T,x,f)$ is dense on $\Bbb C$.

\ \par

The paper is organized as follows: In Section 2, we prove the optimal asymptotic behavior of $\| T^n\|$ for absolutely Cesàro bounded operators and for uniformly Kreiss bounded operators. In particular,  we prove that, for any $0< \varepsilon <\frac{1}{p}$, there exists an absolutely Ces\`{a}ro bounded mixing operator $T$ on $\ell^p(\mathbb{N})$, $1\le p < \infty$, with $\|T^n\|= (n+1)^{\frac{1}{p}-\varepsilon}$. Moreover, we show that
  any absolutely Ces\`{a}ro bounded operator on a Banach space, and any uniformly Kreiss bounded operator on a Hilbert space, satisfies that $\|T^n\|=o(n)$. For  absolutely Ces\`{a}ro bounded operators $T$ on Hilbert spaces we get   $\| T^n\|=o(n^{\frac{1}{2}})$. Section 3 studies ergodic properties of $m$-isometries on finite or infinite dimensional Hilbert spaces. For example,  strict $m$-isometries with $m>3$ are not Cesàro bounded, and we give new examples of weakly ergodic 3-isometries.  In Section 4  we analyze numerical hypercyclicity of $m$-isometries. In particular, we obtain that the adjoint of any strict $m$-isometry unilateral forward weighted shift on  $\ell ^2(\NN)$ is hypercyclic. Moreover, we prove that   weakly ergodic $3$-isometries are weakly numerically hypercyclic.

\section{Absolutely Ces\`{a}ro bounded operators}

It is immediate that any power bounded operator is absolutely Ces\`{a}ro bounded. In general, the converse  is not true.

By $e_n, n\in \NN$, $e_n=(\delta_{n\; k})_{k\in \NN}:=(0, \ldots, 0,1,0,\ldots) $, we denote the standard canonical basis in $\ell^p(\NN)$ for $1\leq p<\infty $.

The following theorem gives a variety of absolutely Ces\`{a}ro bounded operators with different behavior on $\ell^p (\NN)$.

\begin{theorem}\label{ejemplos} Let $T$ be the unilateral weighted backward shift on $\ell ^p(\NN)$ with $1\leq p<\infty$ defined by $Te_1:=0$ and $Te_k:=w_ke_{k-1}$ for $k>1$. If
%\begin{enumerate}
%\item
 $w_k:=\displaystyle \left( \frac{k}{k-1}\right)^{\alpha}  $ with $0<\alpha <\frac{1}{p}$,
 % or
%\item $w_k:=\displaystyle \left( \frac{2k-2}{2k-3}\right)^{1/p}  $,
%\end{enumerate}
then $T$ is absolutely Ces\`{a}ro bounded on $\ell^p(\NN)$.
\end{theorem}

\begin{proof}
%{\it (1)}
Denote $ \varepsilon := 1-\alpha p$. Then $\varepsilon >0$ and $ \alpha =\frac{1-\varepsilon}{p}$.
Fix $x\in \ell^p(\NN)$ with $||x||=1$  given by $x:=\displaystyle \sum_{j=1}^\infty \alpha _j e_j $ and $N\in \NN$. Then

\begin{eqnarray}
 \sum_{n=1}^N \|T^nx\|_p^p &=& \sum_{n=1}^N \sum_{j=n+1}^\infty |\alpha_j|^p\Big(\frac{j}{j-n}\Big)^{1-\varepsilon}   \nonumber \\
 &=& \sum_{j=2}^\infty |\alpha_j|^p\, j^{1-\varepsilon} \sum_{n=1}^{\min\{N,\;j-1\}}({j-n})^{\varepsilon-1} \nonumber \\
 &=& \sum_{j=2}^{2N} |\alpha_j|^p\, j^{1-\varepsilon} \sum_{n=1}^{\min\{ N, \;j-1\}}({j-n})^{\varepsilon-1}
  + \sum_{j=2N+1}^\infty |\alpha_j|^p \sum_{n=1}^N \Big(\frac{j}{j-n}\Big)^{1- \varepsilon}\nonumber \\
 &\le &\sum_{j=2}^{2N} |\alpha_j|^p\, j^{1-\varepsilon}\, \sum_{n=1}^{j-1}(j-1)^{\varepsilon -1} + \sum_{j=2N+1}^\infty |\alpha_j|^p \sum_{n=1}^N \Big(\frac{j}{j-n}\Big)^{1- \varepsilon} \;. \label{des}
 \end{eqnarray}

Notice that for $j>2N$ and $n\leq N$, we have that
$$
\left( \frac{j}{j-n} \right)^{1-\varepsilon}\leq 2^{1-\varepsilon}<2\;.
$$
Hence
$$
\sum_{j=2N+1}^\infty |\alpha_j|^p \sum_{n=1}^N \Big(\frac{j}{j-n}\Big)^{1- \varepsilon}<2N\sum_{j=2N+1}^\infty  |\alpha _j|^p\leq 2N \;.
$$
%We claim that
%$$
%\frac{1}{N} \sum_{n=1}^N \| T^nx\| _p \leq \left( 2+\frac{2}{\varepsilon}\right)^{1/p}\;.
%$$
We can estimate the first term of (\ref{des}) in the following way:
\begin{eqnarray*}
\sum_{n=1}^{j-1} (j-n)^{\varepsilon -1} &= & \sum_{n=1}^{j-1} n^{\varepsilon -1} <  1+ \int _1^{j-1} t^{\varepsilon -1}dt \\
&\le& \frac{(j-1)^\varepsilon}{\varepsilon} <\frac{j^\varepsilon }{\varepsilon} \;.
\end{eqnarray*}

Thus
\begin{eqnarray*}
\sum_{n=1}^N \| T^nx\|^p_p
%&\leq &
% \sum_{j=2}^{2N} |\alpha _j|^pj^{1-\varepsilon} \sum_{n=1}^{j-1} (j-n)^{\varepsilon -1} +  \sum_{j=2N+1} ^\infty |\alpha _j|^p \sum_{n=1}^N\left( \frac{j}{j-n}\right)^{1-\varepsilon}\\
&\leq &
\sum_{j=2}^{2N} |\alpha _j|^pj^{1-\varepsilon} \frac{ j^{\varepsilon }}{\varepsilon} + \sum_{j=2N+1} ^\infty |\alpha _j|^p 2 N\\
&=& \sum_{j=2}^{2N} |\alpha _j|^p \frac{j}{\varepsilon} + 2N \sum_{j=2N+1} ^\infty |\alpha _j|^p \\
&\leq & \frac{2N}{\varepsilon} \sum_{j=2}^{2N} |\alpha _j|^p + 2N \sum_{j=2N+1} ^\infty |\alpha _j|^p \\
& \leq & 2N \left( \frac{1}{\varepsilon} +1\right) \;.
\end{eqnarray*}
By Jensen's inequality
$$
\left( \frac{1}{N} \sum_{n=1}^N \|T^nx\|_p\right)^p \leq \frac{1}{N} \sum_{n=1}^N \| T^nx\|_p^p\leq 2 \left( \frac{1}{\varepsilon } +1\right)  \; ,
$$
which yields the result.
%
%\ \par
%{\it (2)} Let $x:=\sum_{j=1}^\infty \alpha _j e_j\in \ell ^p (\NN)$ with $||x||=1$. By definition
%\begin{eqnarray}
% \sum_{n=1}^N \|T^nx\|_p^p &=& \sum_{n=1}^N \sum_{j=n+1}^\infty |\alpha_j|^p (w_{j}\cdots w_{j-n+1})^p \nonumber\\
% &=& \sum_{j=2}^{\infty} |\alpha _j|^ p \sum_{n=1} ^{\min \{ N,\; j-1\}} (w_{j}\cdots w_{j-n+1})^p \nonumber \\
% & = &  \sum_{j=2} ^{N+1} |\alpha _j|^ p \sum_{n=1} ^{j-1} (w_{j}\cdots w_{j-n+1})^p  + \sum_{j=N+2} ^\infty |\alpha _j|^p \sum_{n=1}^N (w_{j}\cdots w_{j-n+1})^p \nonumber\\
% &\leq & \sum_{j=2}^\infty |\alpha _j|^p\left( \sum_{n=1}^N (w_{N+1}\cdots w_{N+2-n})^p \right) \label{pesos}\\
% &\leq & \sum_{j=1}^{\infty}|\alpha _j|^p\left( \sum_{n=1}^N (w_{N+1}\cdots w_{N+2-n})^p \right)\nonumber \\
% &=& \sum_{n=1}^N (w_{N+1}\cdots w_{N+2-n})^p =2N \;. \label{igualdad}
%\end{eqnarray}
%Notice that  inequality (\ref{pesos}) holds since for $j\leq N+1$ we have that
%$$
%\sum_{n=1}^{j-1}(w_{j}\cdots w_{j-n+1})^p\leq \sum_{n=1} ^N  (w_{N+1}\cdots w_{N+2-n})^p
%$$
%%because $w_n^p\geq 1$
%and, for $j\ge N+2$,
%$$
%\sum_{n=1}^N (w_{j}\cdots w_{j-n+1})^p\leq \sum_{n=1} ^N  (w_{N+1}\cdots w_{N+2-n})^p
%$$
%because $(w_n^p)_{n\in \NN} $ is a decreasing sequence. A simple computation shows the last equality of (\ref{igualdad}). Therefore by Jensen's inequality
%$$
%\left( \frac{1}{N}\sum_{n=1}^N\|T^nx\|_p\right)^p\leq \frac{1}{N} \sum_{n=1}^N \| T^nx\|^p_p\leq 2\;,
%$$
%which gives the desired conclusion.
\end{proof}

As consequence of  above theorem, we obtain

\begin{corollary}
There exist   absolutely Ces\`{a}ro bounded operators which are not power bounded.
\end{corollary}
\begin{proof}
It is an immediate consequence of   Theorem \ref{ejemplos}.
\end{proof}

\begin{corollary}
For $1<p<2$, there exist   absolutely Ces\`{a}ro bounded operators which are not  strongly Kreiss bounded on $\ell^p(\NN)$.
\end{corollary}
\begin{proof}
In view of \cite[Remark 3]{Shi}, if  $T$ is a strong  Kreiss bounded operator
then $\|T^n\|\le Cn^{\frac{1}{2}}$. The conclusion follows from part (1) of  Theorem \ref{ejemplos}.
\end{proof}

\begin{corollary}\label{mixing}
Let $1\leq p<\infty$ and $\varepsilon>0$. Then there exists an absolutely Ces\`{a}ro bounded operators $T$ on $\ell^p$ which  is mixing  and $\|T^n\| =(n+1)^{\frac{(1-\varepsilon)}{p}}$ for all $n\in\NN$.
\end{corollary}
\begin{proof}
By part (1)  of Theorem \ref{ejemplos} we have that $T$ is absolutely Cesàro bounded and
\begin{equation}\label{ejl}
\| T^n\|=(n+1)^{\frac{(1-\varepsilon)}{p}} \;.
\end{equation}
Moreover by \cite[Theorem 4.8]{GEP11} we have that $T$ is mixing if $\left( \prod_{k=1}^n w_k \right)^{-1} \to 0$ as $n\to \infty $. Indeed
$$
\left( \prod_{k=1}^n w_k \right)^{-1}=\frac{1}{n^\alpha } \to 0\;,
$$
hence $T$ is mixing.
\end{proof}

Further consequences can be obtained for operators on Hilbert spaces.

\begin{corollary}
There exists a uniformly Kreiss bounded Hilbert space operator that is not absolutely Cesàro bounded.
\end{corollary}
\begin{proof}
Let $H$ be a separable infinite-dimensional Hilbert space with an orthonormal basis $(u_k)_{k\in \NN}$. Let $0<\alpha<1/2$. Let $T\in B(H)$ be defined by  $Tu_{k}:= \left( \frac{k+1}{k}\right)^\alpha  u_{k+1}$. A straightforward computation gives that $T$ is not absolutely Cesàro bounded since $\| T^nu_1\| =(n+1)^\alpha \to\infty$. Note that its adjoint $T^*$ is given by $T^* u_k=\left(\frac{k+1}{k}\right)^\alpha u_{k-1}$ for $k>1$ and $T^*u_1=0$. By Theorem \ref{ejemplos}, $T^*$ is absolutely Cesàro bounded, and hence uniformly Kreiss bounded. Since the uniform Kreiss boundedness is preserved by taking the adjoints, we deduce that $T$ is uniformly Kreiss bounded.
\end{proof}

\bigskip

It is easy to check that
\begin{equation}\label{media}
\frac{T^n}{n+1} = M_n(T)-\frac{n}{n+1} M_{n-1}(T)\;.
\end{equation}
We notice that Cesàro bounded operators satisfy that $\| T^n\|=O(n)$. Moreover, Theorem \ref{ejemplos} gives an example of a uniformly Kreiss bounded operator  on $\ell^1(\NN)$ such that $\| T^n\| =(n+1)^{1-\varepsilon}$ with $0<\varepsilon <1$.

We concentrate now on Question~\ref{pregunta1} for operators on Hilbert spaces.

\begin{theorem}\label{kreiss}
Let $T$ be a uniformly Kreiss bounded operator on a Hilbert space $H$. Then $\lim_{n\to\infty}n^{-1}\|T^n\|=0$.
\end{theorem}

\begin{proof}
Let $C>0$ satisfy $\bigl\|\displaystyle \sum_{j=0}^{N-1}(\la T)^j\bigr\|\le CN$ for all $\la, |\la|=1$ and all $N$.
We need several claims.

\begin{claim}\label{lemma1}
Let $x\in H$, $\|x\|=1$ and $N\in\NN$. Then
$$
\sum_{j=0}^{N-1}\|T^jx\|^2\le C^2N^2.
$$
\end{claim}

\begin{proof}
Consider the normalized Lebesgue measure on the unit circle. We have
$$
C^2N^2\ge
\int_{|\la|=1} \bigl\|(I+\la T+\cdots+ (\la T)^{N-1})x\bigr\|^2 d\la
$$
$$
 =
\sum_{j,k=0}^{N-1}\int_{|\la|=1}\bigl\langle (\la T)^jx,(\la T)^kx\bigr\rangle d\la=
\sum_{j=0}^{N-1}\int_{|\la|=1}\bigl\langle (\la T)^jx,(\la T)^jx\bigr\rangle d\la=
\sum_{j=0}^{N-1}\|T^jx\|^2.
$$
\end{proof}

\begin{claim}\label{lemma2}
Let $0<M<N$ and $x\in H$, $\|x\|=1$. Then
$$
\sum_{j=0}^{M-1}\frac{\|T^Nx\|^2}{\|T^{N-j}x\|^2}\le C^2M^2.
$$
\end{claim}

\begin{proof}
Set $y=T^Nx$. Since $T^*$ is also uniformly Kreiss bounded, we have
$$
\int_{|\la|=1} \bigl\|(I+(\bar\la T^*)+\cdots+(\bar\la T^*)^{M-1})y\bigr\|^2 d\la\le C^2M^2\|y\|^2.
$$
On the other hand, as in Claim \ref{lemma1} we have
$$
\int_{|\la|=1} \bigl\|(I+(\bar\la T^*)+\cdots+(\bar\la T^*)^{M-1})y\bigr\|^2 d\la=
\sum_{j=0}^{M-1}\|T^{*j}y\|^2
$$
$$
\ge
\sum_{j=0}^{M-1}\Bigl|\Bigl\langle T^{*j}y,\frac{T^{N-j}x}{\|T^{N-j}x\|}\Bigr\rangle\Bigr|^2=
\sum_{j=0}^{M-1}\Bigl|\Bigl\langle y, \frac{T^Nx}{\|T^{N-j}x\|}\Bigr\rangle\Bigr|^2
\ge\|y\|^2 \sum_{j=0}^{M-1}\frac{\|T^Nx\|^2}{\|T^{N-j}x\|^2}.
$$
Hence
$$
\sum_{j=0}^{M-1}\frac{\|T^Nx\|^2}{\|T^{N-j}x\|^2}\le C^2M^2.
$$
\end{proof}

\begin{claim}\label{lemma3}
Let $x\in H$, $\|x\|=1$ and $N\in\NN$. Then
$$
\sum_{j=0}^{N-1}\frac{1}{\|T^jx\|}\ge \frac{\sqrt{N}}{C}.
$$
\end{claim}

\begin{proof}
Let $a_j=\|T^jx\|$. By Claim \ref{lemma1}, $\sum_{j=0}^{N-1}a_j^2\le C^2N^2$.
So
$$
\sum_{j=1}^{N-1}a_j\le\Bigl(\sum_{j=0}^{N-1} a_j^2\Bigr)^{1/2}\cdot\sqrt {N}\le
CN^{3/2}.
$$
Let $B=N\Bigl(\sum_{j=0}^{N-1}\frac{1}{a_j}\Bigr)^{-1}$ and $A=N^{-1}\sum_{j=0}^{N-1}a_j$ be the harmonic and arithmetic means of $a_j$'s for $j\in\{ 0, \ldots , N-1\}$, respectively. By the well-known inequality between these two means, we have
$$
\sum_{j=0}^{N-1}\frac{1}{\|T^jx\|}=
\frac{N}{B}\ge
\frac{N}{A}=
N^2\Bigl(\sum_{j=0}^{N-1}a_j\Bigr)^{-1}\ge
\frac{N^2}{CN^{3/2}}=
\frac{\sqrt{N}}{C}.
$$
\end{proof}

\begin{claim}\label{lemma4}
Let $0<M_1<M_2<N$ and $\|x\|=1$. Then
$$
\sum_{j=M_1}^{M_2-1}\frac{\|T^{N-j}x\|^2}{\|T^Nx\|^2}\ge
\frac{(M_2-M_1)^2}{C^2M_2^2}.
$$
\end{claim}

\begin{proof}
Let $a_j=\frac{\|T^{N-j}x\|^2}{\|T^Nx\|^2}$. By Claim \ref{lemma2},
$$
\sum_{j=M_1}^{M_2-1}\frac{1}{a_j}\le
\sum_{j=0}^{M_2-1}\frac{1}{a_j}\le C^2M_2^2.
$$
Let $A$ and $B$ be the arithmetic and harmonic mean of $a_j$'s for $j\in \{ M_1, \ldots , M_2-1\}$, respectively. We have
$$
\sum_{j=M_1}^{M_2-1} a_j=
(M_2-M_1)A\ge
(M_2-M_1)B=
(M_2-M_1)^2\Bigl(\sum_{j=M_1}^{M_2-1}\frac{1}{a_j}\Bigr)^{-1}\ge
\frac{(M_2-M_1)^2}{C^2M_2^2}.
$$
\end{proof}

\noindent{\it Proof of  Theorem \ref{kreiss}.}
Suppose on the contrary that $\limsup_{n\to\infty}n^{-1}\|T^n\|>c>0$.

Choose $K>8C^6c^{-2}$. Find $N>2^{K+1}$ with $\|T^N\|>cN$ and $x\in H$, $\|x\|=1$ with
$$
\|T^Nx\|>cN.
$$
For $|\la|=1$ let $y_\la=\sum_{j=0}^{N-1}\frac{(\la T)^jx}{\|T^jx\|}$.
Then
$$
\int_{|\la|=1}\|y_{\la}\|^2 d\la=N
$$
and
$$
\int_{|\la|=1}\bigl\|(I+\la T+\cdots+(\la T)^{N-1})y_{\la}\bigr\|^2 d\la\le
C^2N^2
\int_{|\la|=1}\|y_{\la}\|^2 d\la=C^2N^3.
$$

On the other hand,
$$
\int_{|\la|=1}\bigl\|(I+\la T+\cdots+(\la T)^{N-1})y_{\la}\bigr\|^2 d\la
$$
$$
=
\int_{|\la|=1}\Bigl\|\sum_{j=0}^{2N-2} (\la T)^jx \sum_{r=0}^{\min\{N-1,j\}}\frac{1}{\|T^rx\|}\Bigr\|^2 d\la
$$
$$
=
\sum_{j=0}^{2N-2} \|T^jx\|^2 \Bigl(\sum_{r=0}^{\min\{N-1,j\}}\frac{1}{\|T^rx\|}\Bigr)^2\ge
\sum_{j=N-2^K}^{N} \|T^jx\|^2 \Bigl(\sum_{r=0}^{N-2^K}\frac{1}{\|T^rx\|}\Bigr)^2,
$$
where
$$
\sum_{r=0}^{N-2^K}\frac{1}{\|T^rx\|}\ge \frac{\sqrt{N-2^K}}{C}\ge\frac{\sqrt{N}}{C\sqrt{2}}
$$
and
$$
\sum_{j=N-2^K}^{N} \|T^jx\|^2\ge
\|T^Nx\|^2\sum_{k=0}^{K-1}  \sum_{j=N-2^{k+1}}^{N-2^k-1} \frac{\|T^jx\|^2}{\|T^Nx\|^2}\ge
c^2N^2\sum_{k=0}^{K-1} \frac{2^{2k}}{C^2 2^{2k+2}}=
\frac{c^2N^2K}{4C^2}.
$$
Hence
$$
\int_{|\la|=1}\bigl\|(I+\la T+\cdots+(\la T)^{N-1})y_{\la}\bigr\|^2 d\la\ge
\frac{c^2N^2K}{4C^2}\cdot\frac{N}{2C^2}=\frac{c^2KN^3}{8C^4}> C^2N^3,
$$
a contradiction. This finishes the proof.

\end{proof}

\begin{corollary} Any uniformly Kreiss bounded  operator on a Hilbert space is mean ergodic.
\end{corollary}

We are interested on the behavior  of $\frac{\| T^n\|}{n}$ when $T$ is an absolutely Cesàro bounded operator. The following result provides an answer.

\begin{theorem}\label{residual}
Let $X$ be a Banach space, $C>0$ and let $T\in B(X)$ satisfy $\|T^n\|\le Cn$ for all $n\in\NN$. Then either $\displaystyle \lim_{n\to\infty}n^{-1}\|T^n\|=0$ or the set
$$
\Bigl\{x\in X: \sup_N N^{-1}\sum_{n=1}^N\|T^nx\|=\infty\Bigr\}
$$
is residual in $X$.
\end{theorem}

\begin{proof}
Suppose that $\frac{\|T^n\|}{n}\not\to 0$. So there exists $c>0$ such that
$$
\limsup_{n\to\infty}n^{-1}\|T^n\|>c.
$$

For $s\in\NN$ let
$$
M_s=\Bigl\{x\in X: \sup_N N^{-1}\sum_{n=1}^N\|T^nx\|>s\Bigr\}.
$$
Clearly $M_s$ is open.

We show first that each $M_s$ contains a unit vector. Let $s\in\NN$. Find $N>\exp\Bigl(\frac{Cs}{c}\Bigr)+1$ with $\|T^N\|>cN$. Find a unit vector $x\in X$ such that
$\|T^Nx\|> cN$.

For $k=1,\dots,N-1$ we have $\|T^Nx\|\le \|T^k\|\cdot\|T^{N-k}x\|$, and so
$$
\|T^{N-k}x\|\ge\frac{\|T^Nx\|}{\|T^k\|}\ge\frac{cN}{Ck}.
$$
Thus
$$
N^{-1}\sum_{k=1}^N\|T^jx\|\ge
\sum_{k=1}^{N-1}\frac{c}{Ck}\ge
\frac{c}{C}\ln(N-1)>s,
$$
and so $x\in M_s$.

We show that in fact each $M_s$ is dense.
Fix $s\in\NN$, $y\in X$ and $\e>0$. Let $s'>\frac{s}{\e}$. Find $x\in M_{s'}$, $\|x\|=1$. For each $j\in\NN$ we have
$$
\|T^j(y+\e x)\|+\|T^j(y-\e x)\|\ge
2\e \|T^jx\|.
$$
So
$$
\sup_N N^{-1}\sum_{j=1}^N\|T^j(y+\e x)\|+
\sup_N N^{-1}\sum_{j=1}^N\|T^j(y-\e x)\|\ge
\sup_N \frac{2\e}{N}\sum_{j=1}^N\|T^jx\|>
2\e s'>2s.
$$
Hence either $y+\e x\in M_s$ or $y-\e x\in M_s$. Since $\e>0$ was arbitrary, $M_s$ is dense.

By the Baire category theorem,
$$
\bigcap_{s+1}^\infty M_s=\Bigl\{x\in X: \sup_N N^{-1}\sum_{j=1}^N\|T^jx\|=\infty\Bigr\}
$$
is a residual set.
\end{proof}

\begin{corollary}\label{ACB}
Let $T\in B(X)$ be an absolutely Ces\`{a}ro bounded operator. Then $\displaystyle \lim_{n\to\infty}\frac{\|T^n\|}{n}=0$.
\end{corollary}

\begin{proof}
There exists $C>0$ such that
$$
\|T^nx\|\leq \sum_{k=1}^n \| T^kx\|\leq Cn\|x\|
$$
for all $x\in X$.
By Theorem \ref{residual}, we have that  $\displaystyle\lim_{n\to\infty}\frac{\|T^n\|}{n}=0$, since the second possibility in  Theorem \ref{residual} contradicts to the assumption that $T$ is absolutely Ces\`{a}ro bounded.
\end{proof}

As consequence, we obtain a result that, for operators on Banach spaces, slightly improves Lorch theorem \cite{ABR09}.

\begin{corollary} Any absolutely Ces\`{a}ro bounded operator on a reflexive Banach space is mean ergodic.
\end{corollary}

Hence by Corollary \ref{mixing}, we have that

\begin{corollary}
There exist   mean ergodic and mixing operators on $\ell^p(\mathbb{N})$ for $1< p <\infty$ .
\end{corollary}

It is worth to mention that results of this type already appear in the PhD Thesis of Mar\'{\i}a Jos\'e Beltr\'an Meneu \cite{beltran14}, provided by the fourth author (see Section 3.7 in \cite{beltran14}), and in \cite{AS}.

For $0<\varepsilon<1$, by Theorem \ref{ejemplos} we have  an example of absolutely Ces\`{a}ro bounded operators on $\ell^2(\mathbb{N})$
such that $\|T^n\|= (n+1)^{\frac{1}{2}-\varepsilon}$. On the other hand,
if there exists $\varepsilon >0$ such that   $\|T^n\|\ge Cn^{\frac{1}{2}+\varepsilon}$ for all $ n$ in a Hilbert space,
then  by \cite[Theorem 3]{MV}, there exists $x\in X$ such that $\|T^nx\|\rightarrow \infty$, thus $T$ is not
absolutely Ces\`{a}ro bounded. Hence it is natural to ask: does every absolutely Ces\`{a}ro bounded operator on a Hilbert space satisfy  $\lim_{n\to\infty} n^{-1/2}\|T^n\|=0$?

\begin{theorem} \label{acbhilbert}
Let $H$ be a Hilbert space and let $T\in B(H)$ be an absolutely Ces\`{a}ro bounded operator. Then $\displaystyle \lim_{n\to\infty}\displaystyle \frac{ \|T^n\|}{n^{1/2}}=0$.
\end{theorem}

\begin{proof}
Let $C>0$ satisfy $N^{-1}\sum_{n=0}^{N-1}\|T^nx\|<C\|x\|$ for all $N\in\NN$ and $x\in H$.

Suppose on the contrary that $\limsup_{n\to\infty} N^{-1/2}\|T^n\|>0$.
We distinguish two cases:

\medskip
\leftline{{\it Case I}. Suppose that $\limsup_{n\to\infty} n^{-1/2}\|T^n\|=\infty$.}

Then there exist positive integers $N_1<N_2<\cdots$ and positive constants $K_1<K_2<\cdots$ with $\lim_{m\to\infty}K_m=\infty$ such that $\|T^{N_m}\|>K_m N_m^{1/2}$ and
$$
\|T^j\|\le 2K_m j^{1/2}\qquad(j\le N_m).
$$
Let $x_m\in H$ be a unit vector satisfying $\|T^{N_m}x_m\|> K_m N_m^{1/2}$.

Let $N_m'=\Bigl[\frac{N_m}{6}\Bigr]$ \ \ (the integer part).
Consider the set
$$
\{\|T^jx_m\|: 2N_m'\le j< 4N_m'\}.
$$
Let $A$ be the median of this set. More precisely, we have
$$
\card\{j: 2N_m'\le j< 4N_m', \|T^jx_m\|\ge A\} \ge N_m'\qquad\hbox{and}
$$
$$
\card\{j: 2N_m'\le j< 4N_m', \|T^jx_m\|\le A\} \ge N_m'.
$$
We have
$$
4N_m' C\ge
\sum_{j=0}^{4N_m'-1}\|T^jx_m\|\ge
\sum_{j=2N_m'}^{4N_m'-1}\|T^jx_m\|\ge
N_m'A.
$$
So $A\le 4C$ \ \ (note that this estimate does not depend on $m$).

For $\la\in\CC$, $|\la|=1$ let
$$
y_{m,\la}=\sum_{j=1}^{N_m}\frac{(\la T)^jx_m}{\|T^jx_m\|}.
$$
Then
$$
\int\|y_{m,\la}\|^2 d\la=
\int\sum_{j,j'=1}^{N_m} \frac{\langle \la^j T^jx_m, \la^{j'} T^{j'}x_m\rangle}
{\|T^jx_m\|\cdot\|T^{j'}x_m\|} d\la
$$
$$
=
\int\sum_{j=1}^{N_m} \frac{\langle T^jx_m, T^{j} x_m\rangle}
{\|T^jx_m\|^2} d\la=N_m.
$$
Let
$$
u_{m,\la}=(I+\la T+\cdots+ (\la T)^{N_m-1})y_{m,\la}.
$$
Then $\|u_{m,\la}\|\le CN_m\|y_{m,\la}\|$ and
$$
\int\|u_{m,\la}\|^2 d\la\le
C^2N_m^2\int\|y_{m,\la}\|^2 d\la=
C^2N_m^3.
$$
On the other hand,
$$
u_{m,\la}=\sum_{j=1}^{N_m} (\la T)^j x_m \sum_{k=1}^j
\frac{1}{\|T^kx_m\|}+
 \sum_{j=N_m+1}^{2N_m-1}(\la T)^jx_m\sum_{k=j-N_m+1}^{N_m}\frac{1}{\|T^kx_m\|}.
$$
As above,
$$
\int \|u_{m,\la}\|^2 d\la\ge
\sum_{j=1}^{N_m} \|T^j x_m\|^2\Bigl(\sum_{k=1}^{j}\frac{1}{\|T^kx_m\|}\Bigr)^2
\ge
\|T^{N_m}x_m\|^2\Bigl(\sum_{k=2N_m'}^{4N_m'-1}\frac{1}{\|T^kx_m\|}\Bigr)^2
$$
$$
\ge
K_m^2 N_m \cdot \Bigl(\frac{N'_m}{A}\Bigr)^2\ge
K_m^2\cdot{\rm const}\cdot N_m^3.
$$
Since $K_m\to\infty$, this is a contradiction.

\medskip
{\it Case II.} Let $K$ satisfy $0<K<\limsup_{n\to\infty}n^{-1/2}\|T^n\|<2K$.

Let $N_0$ satisfy $n^{-1/2}\|T^n\|\le 2K\quad(n\ge N_0)$.
Find an increasing sequence $(N_m)$ of positive integers such that
$\|T^{N_m}\|>KN_m^{1/2}$. Find $x_m$, $\|x_m\|=1$ such that $\|T^{N_m}x_m\|>KN_m^{1/2}$.

As in case I, let $N_m'=\Bigl[\frac{N_m}{6}\Bigr]$ and let $A$ be the median of the set
$$
\{\|T^jx_m\|: 2N_m'\le j< 4N_m'\}.
$$
Again one has $A\le 4C$.

As in case I, for $|\la|=1$ let
$$
y_{m,\la}=\sum_{j=1}^{N_m}\frac{(\la T)^jx_m}{\|T^jx_m\|}
$$
and
$$
u_{m,\la}=(I+\la T+\cdots+ (\la T)^{N_m-1})y_{m,\la}.
$$
Again we have $\displaystyle\int\|y_{m,\la}\|^2 d\la=N_m$ and
$$
\int\|u_{m,\la}\|^2 d\la\le
C^2N_m^3.
$$
On the other hand,
$$
u_{m,\la}=
\sum_{j=1}^{N_m} (\la T)^j x_m \sum_{k=1}^j
\frac{1}{\|T^kx_m\|}+
 \sum_{j=N_m+1}^{2N_m-1}(\la T)^jx_m\sum_{k=j-N_m+1}^{N_m}\frac{1}{\|T^kx_m\|}
$$
and
$$
\int \|u_{m,\la}\|^2 d\la \ge
\sum_{j=1}^{N_m} \|T^j x_m\|^2\Bigl(\sum_{k=1}^{j}\frac{1}{\|T^kx_m\|}\Bigr)^2
\ge
\sum_{j=4N_m'}^{N_m-1} \|T^j x_m\|^2\Bigl(\sum_{k=2N_m'}^{4N_m'-1}\frac{1}{\|T^kx_m\|}\Bigr)^2
$$
$$
\ge
\sum_{j=4N_m'}^{N_m-1} \|T^j x_m\|^2\Bigl(\frac{N_m'}{A}\Bigr)^2.
$$
Moreover, for $4N_m'\le j< N_m$ we have
$$
KN_m^{1/2}<
\|T^{N_m}x_m\|\le
\|T^{N_m-j}\|\cdot\|T^jx_m\|\le
2K(N_m-j)^{1/2}\|T^jx_m\|.
$$
So
$$
\sum_{j=4N_m'}^{N_m} \|T^j x_m\|^2\ge
\sum_{j=4N_m'}^{N_m-1} \frac{N_m}{4(N_m-j)}\ge
\frac{N_m}{4}\sum_{j=1}^{2N_m'}\frac{1}{j}\ge
\frac{N_m\ln{(2N_m')}}{4}.
$$
Hence
$$
\int\|u_{m,\la}\|^2 d\la\ge
{\rm const} \cdot N_m^3 \ln{(2N_m')},
$$
a contradiction.
\end{proof}

The following picture summarizes the implications between the properties studied here and the behaviour of $\|T^n\|$.

\begin{center}
\begin{figure}[h]
\centering
%\hspace*{1cm}
\begin{tikzpicture}[scale=0.3]
    \node at (20,10) {absolutely Ces\`{a}ro bounded};
  \node at (0,10) {Uniformly Kreiss bounded};
    \node at (0,0) {\mbox {$\|T^n\|= o(n)$}};
    \node at (14,0) {\mbox {$\|T^n\|= o(n)$}};
    \node at (26,0) {\mbox {$\|T^n\|= o(n^{1/2})$}};
     \draw[double, <-] (7.9,10) -- (12,10);
  \draw[double, ->] (0,7) -- (0,2);
     \draw[double, ->] (18,7) -- (14,2);
    \draw[double, ->] (22,7) -- (26,2);
      \node[left] at (0,4.5) {Hilbert space};
        \node[left] at (16,4.5) {Banach space };
          \node[right] at (24,4.5) {Hilbert space};
     \end{tikzpicture}
\caption{Behavior of $\| T^n\|$ for uniformly Kreiss and Cesàro bounded operators. }
\end{figure}
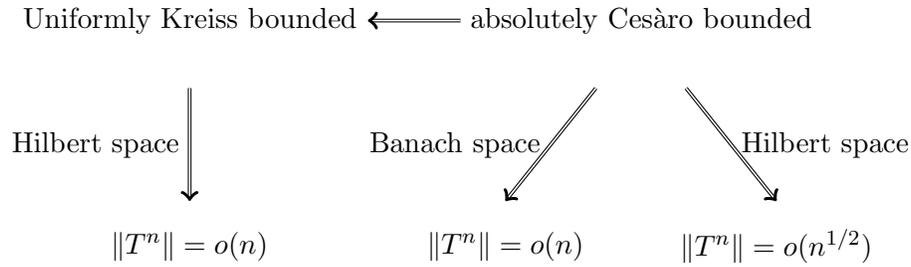
\end{center}

We finish this section with a couple of questions.

\begin{question}
Are there    absolutely Ces\`{a}ro bounded operators on Hilbert  spaces which are not strongly Kreiss bounded?
\end{question}

\begin{question}
Are  there  strongly Kreiss bounded operators which are not  absolutely Ces\`{a}ro bounded?
\end{question}

\section{Ergodic properties for  $m$-isometries}

The following implications for operators on reflexive Banach spaces among various concepts in ergodic theory are a direct consequence of the corresponding definitions:

\begin{center}
\begin{figure}[h]\label{figure2}
\centering
%\hspace*{1cm}
\begin{tikzpicture}[scale=0.3]
  \node at (5.7,0) {Power bounded};
 \node at (16.5,0) {Mean ergodic};
 \node at (16.5,-5) {$\left\| \frac{T^nx}{n}\right\| \to 0 \;\;\; \forall x\in H$};
  \node at (26.5,0) {Weakly ergodic};
   \node at (37,0) {Cesàro bounded};
%\draw[double, ->] (-1,0) -- (0.5,0);
\draw[double, ->] (11,0) -- (12.5,0);
\draw[double, ->] (20.5,0) -- (22,0);
\draw[double, ->] (31,0) -- (32.5,0);
\draw[double, ->] (16.5,-1) -- (16.5,-3.5);
\end{tikzpicture}
\caption{Behavior between different definitions in ergodic theory.}
\end{figure}
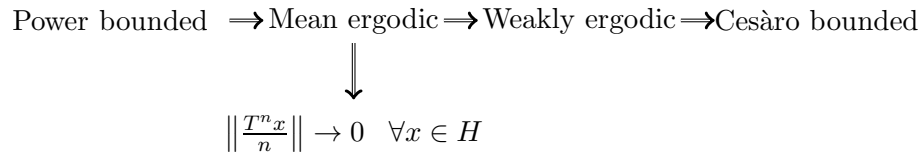
\end{center}

In general, the converse implications of the above  figure  are not true.

\ \par

The purpose of this section is to study $m$-isometries within the framework of these definitions. It is clear that isometries (1-isometries)  are power bounded. It is natural to ask about strict $m$-isometries and the definitions of Figure \ref{figure2} on finite or infinite Hilbert spaces.

The following example is due to Assani. See \cite[page 10]{E} and \cite[Theorem 5.4]{AS} for more details.

\begin{example}\label{ejemplo1}
{\rm Let $H$ be $\RR^2$ or $\CC^2$ and $T=\left(
                                            \begin{array}{cc}
                                              -1 & 2 \\
                                              0 & -1 \\
                                            \end{array}
                                          \right) $.
It is clear that
$$
T^n= \left(
\begin{array}{cc}
 (-1)^n& (-1)^{n-1}2n  \\
0 & (-1)^n  \\
\end{array}
\right) \;
$$
and $\sup_{n\in \NN} \| M_n(T)\| <\infty $. Then $T$ is Cesàro bounded and $\frac{\|T^nx\|}{n}$ does not converge to 0 for some $x\in H$. Hence $T$ is not mean ergodic. Note that $T$ is a strict 3-isometry.
}
\end{example}

The above  example shows that on a 2-dimensional Hilbert space there exists a 3-isometry which is Cesàro bounded and not mean ergodic. This example could be generalized to any Hilbert space of dimension greater or equal to 2.

Let  $H$ be   a  Hilbert space and $T\in B(H)$. Tomilov and Zemánek in \cite{TZ} considered the  Hilbert space ${\cal H}= H\oplus H$ with the norm
$$
\|x_1 \oplus x_2\|_{H\oplus H} = \sqrt{\|x_1\|^2+\|x_2\|^2} \;,
$$
and  the bounded linear operator ${\cal T}$ on  $\cal H$ given by the matrix
\[
{\cal T}:= \left(
\begin{array}{cc}
 T& T-I  \\
0 & T  \\
\end{array}
\right) \;.
\]

In fact, they obtained the following relations of ergodic properties between  the operators $\mathcal{T}$ and $T$.

\begin{lemma}\label{lema1}\cite[Lemmma 2.1]{TZ}
Let $T\in B(H)$. Then
\begin{enumerate}
\item $\mathcal{T}$ is Cesàro bounded   if and only if $T$ is power bounded.
\item $\mathcal{T}$ is mean ergodic  if and only if  $T^n$ converges in the strong topology of $H$.
\item $\mathcal{T}$ is weakly ergodic  if and only if $T^n$ converges in the weak topology of $H$.
\end{enumerate}
\end{lemma}

Recall some properties of $m$-isometries.

\begin{lemma}\label{lema2}
Let $T\in B(H)$ and $m\in \NN$. Then
\begin{enumerate}
\item \cite[Theorem 2.1]{BMNe} $T$ is a strict $m$-isometry if and only if  $\| T^nx\|^2$ is a polynomial at $n$ of degree less or equal to $m-1$ for all $x\in H$, and there exists $x_m\in H$ such that $\| T^nx_m\|^2 $ is a polynomial of degree exactly $m-1$.
\item \cite[Theorem 2.7]{BMNo} If $H$ is a finite dimensional Hilbert space, then $T$ is a strict $m$-isometry with odd $m$ if and only if  there exist a unitary $U\in B(H)$ and a nilpotent operator $Q\in B(H)$ of order $\frac{m+1}{2}$ such that $UQ=QU$ with $T=U+Q$.
\item \cite[Theorem 2.2]{BMNo} If $A\in B(H)$ is an isometry and $Q\in B(H)$ is a nilpotent operator of order $n$ such that commutes with $A$, then $A+Q$ is a strict $(2n-1)$-isometry.
\end{enumerate}
\end{lemma}

\begin{example}\label{ejemplo2}
{\rm Let  $H$ be a  Hilbert space and $T\in B(H)$ such that $T=I+Q$ where $Q^n=0$ for some $n\geq 2$ and $Q^{n-1}\neq 0$. Define the  Hilbert space ${\cal H}$ and  the bounded linear operator ${\cal T}$ on  $\cal H$  as above.
By construction  $\mathcal{T}= A+\mathcal{Q}$ where
$$
A:=\left(
    \begin{array}{cc}
      I & 0 \\
      0 & I \\
    \end{array}
  \right) \;, \;\;\;\;\; \mathcal{Q}:=\left(
                              \begin{array}{cc}
                                Q & Q \\
                                0 & Q \\
                              \end{array}
                            \right)
$$
where $\mathcal{Q}^n=0$ and $\mathcal{Q}^{n-1} \neq 0$. By parts (3) and (1) of Lemma \ref{lema2}, $T$ is a strict $(2n-1)$-isometry and hence not power bounded. Thus, by Lemma \ref{lema1} we have that $\mathcal{T}$ is not Cesàro bounded. It is also simple to verify that $\mathcal{T}$ is strict $(2n-1)$- isometry by Lemma \ref{lema2}.
}
\end{example}

\begin{example}\label{ejemplo3}
{\rm Let $\lambda $ be a unimodular complex number different from 1. Then
$$
{\cal A}:= \left(
\begin{array}{lc}
 \la& \la-1  \\
0 & \la \\
\end{array}
\right)
$$
is a Cesàro bounded operator  (since $\sup_n|\la^n|<\infty$), it is not mean ergodic (since $\la^nx$ does not converge) and is a $3$-isometry on $\CC^2$, see Lemmas \ref{lema1} and \ref{lema2}.
}
\end{example}

Now we give some ergodic properties of $m$-isometries.

Example \ref{ejemplo1} is   a Ces\`{a}ro bounded  $3$-isometries. However, as a consequence of Theorem \ref{kreiss} and Lemma \ref{lema2}, we obtain the following.

\begin{corollary}
There is no  uniformly Kreiss bounded strict $3$-isometry.
\end{corollary}

\begin{theorem}\label{Cesarobounded}
 Assume that $H$ is a finite $n$-dimensional Hilbert space. Then
\begin{enumerate}
\item If $n\geq 2$, then there exists a Cesáro bounded strict 3-isometry.
\item The isometries are the only mean ergodic strict $m$-isometries on $H$.
\end{enumerate}
\end{theorem}
\begin{proof}
{\it (1)}
Let
$$
{\cal A}:= \left(
\begin{array}{lc}
 \la& \la-1  \\
0 & \la \\
\end{array}
\right)
$$
be the operator  on $\CC^2$ considered in Example \ref{ejemplo3}.
Write $H=\CC^2\oplus\CC^{n-2}$ and let ${\cal B}: = {\cal A}\oplus I_{\CC^{n-2}}$. Then $\cal B$ is a strict $3$-isometry which is Ces\`{a}ro bounded (and not power bounded).

{\it (2)} Suppose that $T$ is a strict $m$-isometry with $m>1$ on a finite dimensional Hilbert space, then $m\ge 3$. Using part (1) of Lemma \ref{lema2}, it is easy to prove  that $\frac{\|T^nx\|}{n}$ does not converges to 0 for some  $x\in H$. So, $T$ is not mean ergodic.
\end{proof}

In infinite dimensional Hilbert space we can say more.

\begin{theorem} Let $T$ be a strict $m$-isometry. Then
 \begin{enumerate}
\item  If $m>3$, then $T$ is not Cesàro bounded. In particular there is no weakly ergodic  strict $m$-isometry for $m>3$.
\item If $m\geq 3$, then $T$ is not mean ergodic.
\end{enumerate}
\end{theorem}

\begin{proof}
By part (1) of Lemma \ref{lema2}, there exists $x\in H$ such that $\| T^nx\|^2$ is a polynomial at $n$ of order  $m-1$ exactly. Thus by equation (\ref{media}), the proof is complete.
%Notice that
%$$
%\frac{T^n}{n+1} = M_n(T) -\frac{n}{n+1} M_{n-1}(T)\;.
%$$
%Thus if $T$ is Ces\`{a}ro bounded (mean ergodic) then $ \frac{1}{n} T^n$ is bounded (tends to zero) for all $x$. Thus there are not  Ces\`{a}ro bounded (mean ergodic) $m$-isometries  if $m>3$ ( $m\ge 3$).
%
%Moreover if $ \{ M_n(T)x\}$ converges weakly, then $\frac{1}{n} T^nx $ converges weakly and thus it must  be bounded.
\end{proof}

%Recall that a complex measure    $\mu $ on the unit disk, $\Bbb T$, is called  \emph{Rajchman measure} if
%$$
%(\hat\mu)(n)= \int_{\Bbb T}z^n d\mu\rightarrow 0
%$$
%as $|n| \rightarrow \infty$. Notice that  every absolutely continuous measure is Rajchman.

%Denote
%$$
%L^2(\Bbb T,\mu) := \{ f: \Bbb T \rightarrow \Bbb C : \int_{\Bbb T}|f|^2 d\mu <\infty \} \;,
%$$
%and $M_z$ the multiplication operator  on $L^2(\Bbb T,\mu)$.
% is weakly convergent to zero  if and only if $\mu $ is Rajchman \cite{K}.

\begin{theorem}\label{Rachjman}
%For every  Rachjman measure $\mu $ on $\Bbb T$, there is a Ces\`{a}ro bounded and weakly ergodic strict $3$-isometry on  $L^2(\Bbb T,\mu)\oplus L^2(\Bbb T,\mu)$.
There exists a Ces\`{a}ro bounded and weakly ergodic strict $3$-isometry.
\end{theorem}

\begin{proof}
Let $U$ be the bilateral shift.
Define
$$
{\cal M} : = \left(
\begin{array}{cc}
 U& U-I  \\
0 & U\\
\end{array}
\right) \;.
$$
First observe that ${\cal M}$ is Ces\`{a}ro bounded, by part (1) of Lemma \ref{lema1}. Since $U^n\to 0$ in the weak operator topology, $\cal M$ is weakly ergodic by part (3) of Lemma \ref{lema1}. Therefore, the conclusion is derived by part (3) of Lemma \ref{lema2}.
\end{proof}

In \cite{AS}, it is given an example  of  a Ces\`{a}ro bounded strict $3$-isometry $T$ on a  Hilbert space $H$ for which the sequence  $\left(\displaystyle \frac{ T^n}{n}\right)_{n\in \NN}$ is bounded below for all $x\in H \setminus \{0\}$. In particular, $ \left( M_n(T)x\right)_{n\in \NN}$ diverges for each $x\in H \setminus \{0\}$, and  it is weakly ergodic.

%\begin{corollary}
%Given  a separable infinite dimensional  Hilbert space $H$, there are weakly ergodic  strict 3-isometries $T$ on $H$.
%\end{corollary}
%\begin{proof}
%Let $ M_z$ be the operator  given in \cite[Theorem 5.4]{AS} on ${H}_1$.  Consider $T:=U^*{\cal M}_z U$ where $U$ is a unitary isomorphism between $H$  and ${ H}_1$. It is straightforward to see that $m$-isometries and weak ergodic operators are preserved by conjugation with unitary operators. this completes the proof.
%\end{proof}

%The example of  $3$-isometry  $T$ given in \cite{AS} satisfies that the sequence  $\left(\displaystyle \frac{ T^n}{n}\right)_{n\in \NN}$ is bounded below for all $x\in H \setminus \{0\}$.
%Indeed this property can be characterized  for 3-isometries.
We give a characterization of this property.

Given an $m$-isometry $T$,  the \emph{covariance operator} of $T$ is defined by
$$
\Delta_T: = \frac{1}{k!} \sum _{j=0} ^m  (-1)^{m-j} \binom  {m} {j} {T^*}^j T^j\; .
$$

\begin{theorem}
  Let $T$  be  a  strict $3$-isometry   on a Hilbert space $H$.  Then  the sequence  $\left(\displaystyle \frac{ T^nx}{n}\right)_{n\in \NN}$ is bounded below for all $x\in H \setminus \{0\}$ if and only if the covariance operator $\Delta_T$ is injective.
\end{theorem}

\begin{proof}
If $T$  is a  strict $3$-isometry and   $\Delta_T$ is injective, then
$\displaystyle \inf _{n}\frac{\|T^nx\|}{n} >0 $ for all $x\in H\setminus \{0\}$ (see the  proof of \cite[Theorem 3.4]{BMM}).

If $\Delta_T$ is not injective, then there exists $x$ such that $\langle \Delta_Tx,x \rangle =0$. By  \cite[Proposition 2.3]{BMM}, we have that $\displaystyle \inf _{n}\frac{\|T^nx\|}{n} \rightarrow \langle \Delta_Tx,x \rangle =0$, and thus the sequence  $\displaystyle \frac {T^nx}{n}$ is not bounded below.
\end{proof}

There exist weakly ergodic strict $3$-isometries with the covariance operator $\Delta_T$  injective  by \cite[Section 5.2]{AS} and not injective, see the  proof of Theorem  \ref{Rachjman}.

The Uniform ergodic theorem of Lin \cite[Theorem]{Li} asserts  that if  $\displaystyle \frac{\|T^n\|}{n} \rightarrow 0$, then  $T$ is uniformly ergodic if and only if the range of $I-T$ is closed. On the other hand, $T$ is uniformly ergodic if and only if  $\displaystyle \frac{\|T^n\|}{n} \rightarrow 0$ and 1 is a pole of the  resolvent operator.

\begin{corollary}
For $m>1$, there is no uniform ergodic strict $m$-isometry  on a Hilbert space.
\end{corollary}
\begin{proof}
Since there is no mean ergodic strict $m$-isometry  for $m\ge 3$, the result follows immediately from the fact that any strict $2$-isometry $T$ satisfies that the spectrum $\sigma(T)=\overline{\mathbb{D}}$ and, thus, 1 is not an isolated point of $\sigma (T)$.
\end{proof}

There exists a strict $3$-isometry  $T$  which  is weakly ergodic   (thus Ces\`{a}ro bounded), but  it is not mean ergodic. For 2-isometries something else can be established.

\begin{corollary} Let $H$ be an infinite dimensional Hilbert space and let $T$ be a strict 2-isometry. Then the following assertions are equivalent:
\begin{enumerate}

\item  $T$ is  mean ergodic.

\item  $T$ is weakly ergodic.

\item  $T$ is  Ces\`{a}ro bounded.
\end{enumerate}

\end{corollary}
\begin{proof}
It is a consequence of part (1) of Lemma \ref{lema2}, since $\frac{T^nx}{n}$ converges  to zero for all $x\in H$.
\end{proof}

The following example provides a  $2$-isometry that is not Ces\`{a}ro bounded.

\begin{example}
On $\ell^2(\NN)$ we consider  the operator $T$ given by  $T(x_1, x_2,\ldots ): =
(x_1, x_1, x_2,x_3,\ldots )$. Then $T$  is a $2$-isometry which is  not Ces\`{a}ro bounded.
\end{example}

\begin{proposition} \label{Cesaro}
Let $T$ be the weighted backward shift in $\ell^p(\mathbb{N})$ with  $1\le p<\infty$ defined by $Te_1:=0$, $Te_j:=\Bigl(\frac{j}{j-1}\Bigr)^{1/p} e_{j-1}\quad(j>1)$. Then $T$ is not Ces\`{a}ro bounded.
\end{proposition}

\begin{proof}
Let $x_n:=\frac{1}{n^{1/p}}\sum_{s=1}^n e_s$ with even $n$. It is clear that $\|x_n\|_p=1$.
We have
\begin{eqnarray*}
\Bigl\|\frac{1}{n}\sum_{j=0}^{n-1} T^jx_n\Bigr\|^p_p &=&
\frac{1}{n^{p+1}}\Bigl\|\sum_{j=0}^{n-1}\sum_{s=1}^n T^je_s\Bigr\|^p_p=
\frac{1}{n^{p+1}}\Bigl\|\sum_{s=1}^{n}e_s\sum_{j=s}^n \Bigl(\frac{j}{s}\Bigr)^{1/p}\Bigr\|^p_p\\[1pc]
& =&
\frac{1}{n^{p+1}}\sum_{s=1}^{n}\Bigl(\sum_{j=s}^n \Bigl(\frac{j}{s}\Bigr)^{1/p}\Bigr)^p\ge
\frac{1}{n^{p+1}}\sum_{s=1}^{n/2+1}\frac{1}{s}\Bigl(\sum_{j=n/2+1}^{n} j^{1/p}\Bigr)^p,
\end{eqnarray*}
where
$$
\sum_{j=n/2+1}^{n} j^{1/p}\ge\int_{n/2}^n t^{1/p}dt\geq
\frac{1}{p^{-1}+1}\Bigl(n^{1+p^{-1}}-\Bigl(\frac{n}{2}\Bigr)^{1+p^{-1}}\Bigr)=c n^{1+1/p}
$$
with $c=\frac{p}{p+1}(1-\frac{1}{2^{1+p^{-1}}})>0$.
So
$$
\Bigl\|n^{-1}\sum_{j=0}^{n-1} T^jx_n\Bigr\|^p_p\ge
\frac{1}{n^{p+1}}\sum_{s=1}^{n/2} \frac{c^p n^{p+1}}{s}\ge
c^p\ln\frac{n}{2}\to\infty
$$
as $n\to\infty$. Hence $T$ is not Ces\`{a}ro bounded.
\end{proof}

%Recall that the  unilateral weighted shift $S_w$ on $\ell^2(\Bbb N)$ is a strict $2$-isometry if and only if there exists a real number $b<1$ such that $$
%|w_n| =\sqrt{\frac{n+1-b}{n-b}} \;.
%$$

\begin{corollary}
There is no  Ces\`{a}ro bounded weighted forward shift on  $\ell^2(\mathbb{N})$, which is a strict $2$-isometry.
\end{corollary}

\begin{proof}
Assume that $T$ is a weighted forward shift with weights $(w_n)_{n\in \NN}$. By \cite[Theorem 1]{AL} (see also \cite[Remark 3.9]{BMNe}), if $T$ is a strict 2-isometry, then
$$
|w_n|^2=\frac{p(n+1)}{p(n)}\;,
$$
where $p$ is a polynomial of degree 1, that is, $p(n):=an+b$.

First, suppose that $b=0$. Then $w_n=\sqrt{\frac{n}{n-1}}$, since $a\neq 0$. Hence $T^*e_n:=\sqrt{\frac{n}{n-1}}e_{n-1}$.
 By Proposition  \ref{Cesaro}, $T^*$ is not Cesàro bounded. Since  Cesàro boundedness is preserved by taking adjoints, $T$ is not Cesàro bounded.

Now, assume that $b\neq 0$, then $w_n(c):=\sqrt{\frac{cn+1}{c(n-1)+1}}$ with $c\neq 0$. Denote $T_ce_n:=w_n(c)e_{n+1}$ and the diagonal operator $Ve_n:= \alpha_n e_n$, where $\alpha _n:=\sqrt{\frac{c(n-1)+1}{n}}$. Then $V$ is invertible and satisfies that
$VT_1=VT_c$. Moreover, $T_1$ is not Cesàro bounded, by following an argument as in Proposition \ref{Cesaro}. Using  that  Cesàro boundedness is preserved by similarities, we obtain that $T_c$ is not Cesàro bounded.
\end{proof}

\begin{corollary}
There is no absolutely Ces\`{a}ro bounded strict $2$-isometry  on a Hilbert space.
\end{corollary}
\begin{proof}
It is immediate by Theorem \ref{acbhilbert} and part (1) of Lemma \ref{lema2}.
\end{proof}

\begin{question}
Is it possible to construct a Ces\`{a}ro bounded  strict $2$-isometry  on an  infinite dimensional Hilbert space?
\end{question}

\section{ Numerically hypercyclic properties of $m$-isometries}
In this section we study numerically hypercyclic $m$-isometries. For simplicity we discuss only operators on Hilbert spaces.

\begin{definition}
Let $H$ be a Hilbert space. An operator $T\in B(X)$ is called numerically hypercyclic if there exists a unit vector $x\in H$ such that the set $\{\langle T^nx,x\rangle: n\in\NN\}$ is dense in $\CC$.
\end{definition}

Clearly the numerical hypercyclicity is preserved by unitary equivalence but in general not by similarity. This leads to the following definition:

\begin{definition} Let $T\in B(X)$.  It is said  that $T$ is \emph{weakly numerically hypercyclic} if $T$ is similar to a numerically hypercyclic operator.
\end{definition}

In \cite[Proposition 1.5]{S},  Shkarin  proved that  $T\in B(H)$ is weakly numerically hypercyclic if and only if there exist $x,y\in H$ such that the set
$\{\langle T^nx,y\rangle: n\in\NN\}$ is dense in $\mathbb{C}$.

Faghih and Hedayatian  proved in \cite{FaHe} that   $m$-isometries on a Hilbert space are not weakly hypercyclic.  Moreover,  $m$-isometries on a Banach space are not 1-weakly hypercyclic \cite{BBF}.  However, there are isometries that are weakly supercyclic \cite{Sanders05} (in particular cyclic). Thus the first natural question is the following: are
there numerically hypercyclic $m$-isometries?

\begin{theorem}
There are no    weakly numerically hypercyclic $m$-isometries on $B(\mathbb{C}^n)$ for $n\le 3$.
\end{theorem}
\begin{proof}
If $n=1$, there are not weakly numerically hypercyclic operators.
Let  $n=2$. By \cite[Theorem 1.13]{S}, if $T\in B(\mathbb{C}^2)$ is a weakly numerically hypercyclic operator, then there exists $\lambda\in \sigma(T)$, with $|\lambda|>1$ and thus $T$ is not  an $m$-isometry. For $n=3$, it is the same by \cite[Theorem 1.14]{S}.
\end{proof}

We discuss the existence of weakly numerically hypercyclic $m$-isometries on $n$-dimensional spaces for $n\ge 4$.

We say that $\lambda_1,\lambda_2\in \mathbb{T}$ are \emph{rationally independent} if $ \lambda_1^{m_1}\lambda_2^{m_2}\ne 1$ for every non-zero pair $m=(m_1,m_2)\in \mathbb{Z}^2$, or equivalently if $\lambda_j=e^{i\theta_j}$ with $\theta_j\in \mathbb{R}$ with $\pi, \theta_1,\theta_2 $ are linearly independent over the field $\mathbb{Q}$ of rational numbers.

If $T\in B(X)$ and  there are rationally independent $\lambda_1,\lambda_2\in \mathbb{T}$ such that $ker(T-\lambda_jI)^2 \ne ker(T-\lambda_jI)$ for $j\in \{1,2\}$, then $T$ is weakly numerically hypercyclic \cite[Theorem 1.9]{S}. Moreover if  $X$ is a Hilbert space, then  $T$ is numerically hypercyclic \cite[Proposition 1.12]{S}.
The following result gives an answer to the above question for some $m$-isometries.

%\begin{theorem}\label{3-isometry}
%There exists a    numerically hypercyclic strict 3-isometry on $L(\Bbb C^n)$ for $n\ge 4$.
%\end{theorem}
%
%\begin{proof}
%
%By \cite[Proposition 1.12]{S}, it is sufficient to see that there exists   a 3-isometry $T\in L(\Bbb C^n)$ such that there are independent $\lambda_1,\lambda_2\in \Bbb T$ with $ker(T-\lambda_jI)^2 \ne ker(T-\lambda_jI)$ for  $j\in \{1,2\}$.
%
%For  $n=4$, we define $Te_1 := \lambda_1 e_1$, $Te_2 := e_1+ \lambda_1e_2$, $Te_3 := \lambda_2 e_3$ and $Te_4 :=  e_3+ \lambda_2 e_4$. Then $T$ satisfies the above condition. Also $T$ is a 3-isometry because $T=A+Q$ with  $A$ an isometry and $Q$ a nilpotent operator of order 2 such that $AQ=QA$ \cite[Theorem 2.2]{BMNo}.
%
%For $n>4$, we consider $\widetilde{T}e_1 := \lambda_1e_1$, $\widetilde{T}e_2 := e_1+ \lambda_1e_2$, $\widetilde{T}e_3 := \lambda_2 e_3$, $\widetilde{T}e_4 :=  e_3+ \lambda_2 e_4$, and $\widetilde{T}e_i: = e_i$ for $5\le i \le n$. That is $\widetilde{T}:=T\oplus I$ on $\Bbb C^4 \oplus \Bbb C^{n-4}$.
%
%\end{proof}

\begin{theorem}\label{3-isometry}
There exists a numerically hypercyclic strict $(2m-1)$-isometry on $B(\mathbb{C}^n)$, with $n\geq 4$, for $2\leq m\leq n-2$.
\end{theorem}

\begin{proof}
Let $\ell \in \{ 2, 3, \ldots, n-2\}$. We will construct a numerically hypercyclic strict $(2\ell -1)$-isometry. Define $D$ the diagonal operator with diagonal
$$
(\underbrace{\lambda _1, \cdots, \lambda_1}_{\ell }, \lambda _2, \lambda _2, \underbrace{1, \cdots , 1}_{k-2\ell})
$$
where $\lambda _1$ and  $\lambda _2$ are rationally independent complex numbers   with modulus 1 and $Q$  by
\begin{eqnarray*}
Q e_i :&= & e_{i-1} \mbox { for } i\in \{ 2, 3, \cdots, \ell \}\\
Q e_{\ell +2}:&=& e_{\ell +1} \mbox { and }\\
Q e_i:&= &0 \mbox{ for } i=1, i=\ell+1 \mbox{ and  } i\geq \ell +3 \;.
\end{eqnarray*}
It is clear that $Q^\ell =0 $ and $Q^{\ell -1} e_\ell =e_1\neq 0$.
 Moreover,
\begin{eqnarray*}
QDe_i&=& DQe_i= \lambda _1 e_{i-1} \mbox{ for } 2\leq i\leq \ell \\
QDe_{\ell +2} &=& DQe_{\ell +2} =\lambda _2 e_{\ell +1} \\
QDe_i&=&DQe_i=0 \mbox{ for } i=1, \ell+1\mbox{ and }\geq i\ge\ell +3\;.
\end{eqnarray*}
By part (3) of Lemma \ref{lema2}, $T:=D+Q$ is a strict $(2\ell -1)$-isometry for any $\ell \in \{ 2, 3, \cdots, n-2\}$.

Let us prove that $T$ satisfies that $Ker (\lambda_i -T)\neq Ker (\lambda_i -T)^2$ for $i=1,2$. By definition  $e_2\in Ker (\lambda_1 -T)^2\setminus Ker (\lambda_1 -T)$ and $e_{\ell +1} \in Ker (\lambda_2 -T)^2\setminus Ker (\lambda_2 -T)$. So by \cite[Proposition 1.9]{S},  $T$ is numerically hypercyclic.
\end{proof}

As a consequence of the proof of Theorem \ref{3-isometry}, we obtain

\begin{corollary}
Let $H$ be a  complex Hilbert space with dimension at least 4. Then there exists a   numerically hypercyclic strict 3-isometry on H.
\end{corollary}

\begin{theorem}
An $n$-dimensional Hilbert space supports no weakly numerically hypercyclic strict $(2n-3)$ or $(2n-1)$-isometries.
\end{theorem}

\begin{proof}
Let $H$ be a finite-dimensional Hilbert space, $\dim H=n<\infty$. Suppose on the contrary that $T\in B(H)$ is a weakly numerically hypercyclic $(2n-1)$-isometry.
Since $\|T^kx\|^2$ grows polynomially for each $x\in H$ and there exists $u\in H$ such that
$\|T^ku\|^2$ is a  polynomial  of degree $2n-2$, the Jordan form of $T$ has only one block corresponding to an eigenvalue $\la$ with $|\la|=1$. Thus $T=\la I+Q$ where $Q^n=0$. Thus
$$
T^k =\sum_{j=0}^{n-1}\binom{k}{j}\la^{k-j}Q^{j}=
\la^k\sum_{j=0}^k {\binom{k}{j}}\la^{-j}Q^j
$$
for all $k\in\NN$.

Let $x,y\in H$ and suppose that the set $\{\langle T^kx,y\rangle:k\in\NN\}$ is dense in $\CC$. We have
$\langle T^kx,y\rangle= \la^k p(k)$ for some polynomial $p$ of degree $\le n-1$.
If $\deg p\ge 1$ then $|\langle T^kx,y\rangle|\to\infty$ so the set $\{\langle T^kx,y\rangle:k\in\NN\}$ is not dense in $\CC$.

If $\deg p=0$ then the set $\{\langle T^kx,y\rangle:k\in\NN\}$ is bounded and again is not dense in $\CC$.
Hence $T$ is not weakly numerically hypercyclic.

The case of $(2n-3)$-isometries can be treated similarly. If $T\in B(H)$ is a strict $(2n-3)$-isometry then the Jordan form of $T$ has two blocks: one of dimension $n-1$ corresponding to an eigenvalue $\la$, $|\la|=1$ and the second one-dimensional block corresponding to an eigenvalue $\mu, |\mu|=1$. For $x,y\in H$ we have $\langle T^kx,y\rangle=\la^kp(k)+ a\mu^k$ for some polynomial $p, \deg p\le n-2$ and a number $a\in \CC$. Again one can show easily that the set $\{\langle T^kx,y\rangle:k\in\NN\}$ cannot be dense in $\CC$. Hence there are no weakly numerically hypercyclic $(2n-3)$-isometries on $H$.
\end{proof}

%\begin{remark}{\rm
%An $n$-dimensional Hilbert space supports no strict $(2n-3)$ or $(2n-1)$-isometry $T$ such that it satisfies
%\begin{equation}\label{ec}
%Ker (\lambda_i -T)\neq Ker (\lambda_i -T)^2
%\end{equation}
%with $i=1,2$ for rationally independent $\lambda _1$ and $\lambda _2\in \mathbb{T}$.

%Suppose  that $n=4$ and assume that $T$ satisfies (\ref{ec}), then $T$ is similar to
%$$
%T':=\left( \begin{array}{cc}
%\left[ \begin{array}{cc}
%\lambda _1 & 1\\
%0 & \lambda _1
%\end{array} \right] & 0 \\
%0 &
%\left[ \begin{array}{cc}
%\lambda _2 & 1\\
%0 & \lambda _2
%\end{array} \right]
%\end{array}
%\right)
%$$
 %and there exists an invertible $S$ such that $T^n=ST'^n S^{-1}$ for all $n\in \mathbb{N}$. In fact, $T'= D+ Q$, where $D$ is a diagonal operator  with diagonal $(\lambda_1, \lambda_1, \lambda_2, \lambda_2)$ and $Q^2=0$. Hence for all $x\in \mathbb{C}^4$
 %\begin{equation}\label{ecd}
 %\| T^nSx\|=\|ST'^nx\|=\|S(D^n+nQ)x\| \leq \| S\| (\| x\|+n\| Qx\|) \;.
 %\end{equation}
 %In particular, $T$ cannot be a strict $m$-isometry for $m>3$. Otherwise there would exist $y\in \mathbb{C}^4$ such that $\|T^ny\|$ is a polynomial of degree exactly $\frac{m-1}{2}$, and this would contradict (\ref{ecd}).
%}
%\end{remark}

%\begin{question}\label{m-isometrias}
%Let $n\geq 4$. Are there numerically hypercyclic $(2n-1)$ or $(2n-3)$-isometries on $\Bbb C^n$?
%\end{question}

\begin{theorem}
For $m \ge 2$, there exists a  numerically hypercyclic strict $m$-isometry on $\ell^2(\mathbb{N})$.
\end{theorem}
\begin{proof}

For $m\ge 2$, no strict $m$-isometry is  power bounded \cite[Theorem 2]{COT}. Also by \cite[Theorem 1]{AL}, there exist forward weighted shifts on $\ell^2(\mathbb{N})$  that are  strict   $m$-isometries for $m\ge 2$. Now,  using that if $1<p<\infty$ and $T$ is  a forward weighted shift on $\ell^p(\mathbb{N})$, then $T$ is numerically hypercyclic if and only if $T$ is not power bounded (\cite{KPS} \& \cite{S}), we obtain the result.
\end{proof}

%\begin{theorem}
%For $m \ge 3$ and $m$ an odd integer,  there exist strict $m$-isometries numerically hypercyclic on $\ell^2(\Bbb Z)$.
%\end{theorem}
%
%\begin{proof}
%
%The proof is consequence of $m \ge 3$ and $m$ an odd integer, there exist a bilateral weighted shift on $\ell^2(\Bbb Z)$ that are  strict $m$-isometries \cite[Corollary 20]{AL}.
%
%\end{proof}

%Given an infinite dimensional separable complex Hilbert space and  a positive integer $m$, there exists a strict $m$-isometry on $H$ \cite[Proposition 5.1]{BJZ}. Moreover on $\ell^2(\Bbb N)$ and $\ell^2(\Bbb Z)$ there exist numerically hypercyclic $m$-isometries.

%\begin{lemma}
%Let $H_1$ and $H_2$ be Hilbert spaces, let $T$ be a numerically hypercyclic operator on $H_1$, and let $U: H_1 \rightarrow H_2$ be a  unitary operator. Then  $U TU^*$ is a numerically hypercyclic operator on $H_2$.
%\end{lemma}
%\begin{proof}
%Since  $T$ is  a numerically hypercyclic operator on $H_1$ then there exists $x\in H_1$ with $\|x\|=1$ such that $\{ \langle T^nx,x \rangle: n\in \Bbb Z^+\}$ is dense on $\Bbb C$.

%Let $y= Ux \in H_2$, then $\|y\|=1$ and
%$$
%\langle (UTU^*)^ny,y \rangle = \langle (UTU^*)^nUx, Ux \rangle =\langle UT^nx,Ux \rangle =\langle T^nx,x \rangle \;.
%$$
%Then $\{ \langle (UTU^*)^ny,y \rangle: n\in \Bbb Z^+\}$ is dense on $\Bbb C$ and thus $U TU^*$ is a numerically hypercyclic operator on $H_2$.
%\end{proof}

Since  both  numerical hypercyclicity and  $m$-isometricity are properties preserved by unitary equivalence, we have that
\begin{corollary}
Let $H$ be  an infinite dimensional separable complex Hilbert space and $m\ge 2$. Then there exists a  numerically hypercyclic $m$-isometry on $H$.
\end{corollary}

%\begin{question}
%Let $H$ be a Hilbert space  of dimension at least 4. Are there a  numerically hypercyclic  $m$-isometry?
%\end{question}

\begin{theorem}
There exists a   numerically hypercyclic Ces\`{a}ro bounded strict $3$-isometry on $\mathbb{C}^4$.
\end{theorem}
\begin{proof}
Let $T$ be the operator considered in the proof of Theorem \ref{3-isometry}
$$
T:=\left(
\begin{array}{cccc}
 \lambda _1 & \lambda _1 -1 & 0& 0  \\
0& \lambda _1&0 &0 \\
0& 0&\lambda _2 & \lambda _2 -1\\
0&0 & 0& \lambda _2
\end{array}
\right)
\;,
$$
where $\lambda_1,\lambda_2\in \mathbb{T}$ are  rationally independent.
By the proof of Theorem \ref{3-isometry}, it is clear that $T$ is numerically hypercyclic.

Since both blocks
$$
\left(
\begin{array}{cc}
 \lambda _1 & \lambda_1 -1   \\
0& \lambda _1 \\
\end{array}
\right)
\qquad\hbox{and}\qquad
\left(
\begin{array}{cc}
 \lambda _2 & \lambda_2 -1   \\
0& \lambda _2 \\
\end{array}
\right)
$$
are Cesàro bounded by Lemma \ref{lema1},
it is easy to see that $T$ is Cesàro bounded.

%Let us prove that $T$ is Cesàro bounded. Note that
%$$
%T^n= \left(
%\begin{array}{llll}
%  \lambda_1^n&  n \lambda_1^{n-1}(\lambda_1-1) &0 & 0\\
%  0&  \lambda_1^n &0 & 0\\
 % 0&  0 &\lambda_2^n&  n \lambda_2^{n-1}(\lambda_2-1)\\
%0&  0 &0 & \lambda_2^n\\
%\end{array}
%\right) \;.
%$$
%Using that $\frac{1}{n+1} \sum_{k=0}^n \lambda _i ^k$ and $\frac{1}{n+1} \sum_{k=0}^n k \lambda _i ^{k-1} (\lambda _i -1)$ are bounded sequences for $i=1,2$,  we get the result.
\end{proof}

We know that there exist examples of numerically hypercyclic $3$-isometries and  weakly ergodic  $3$-isometries.
 The following result goes further in this direction.

\begin{theorem}
Any weakly ergodic strict $3$-isometry on a Hilbert space  is  weakly numerically hypercyclic.
\end{theorem}
\begin{proof}
If $T$ is a weakly ergodic  strict $3$-isometry, then there exists $x$ such that $\displaystyle \frac{T^nx}{n}$ is weakly convergent  but it is not norm convergent. Indeed  for  a strict $3$-isometry   $T$, there exists $x$ such that  $\displaystyle \frac{T^nx}{n}$ does not converge to zero in norm.

Then, since $\displaystyle x_n =\frac{T^nx}{n}$ is weakly convergent  but it is not norm convergent, by \cite[Lemma 6.1]{S} there is
$y\in H$ such that $\{n \langle x_n,y\rangle: n\in \mathbb{N}\}$ is dense on $\mathbb{C}$.
Hence
%Thus $O(T,x,f)$ is dense on $\Bbb C$ since $\{n f(x_n): n\in \Bbb N\} $ is a subset of $O(T,x,f)$. By \cite[Proposition 1.5]{S}
$T$ is weakly numerically hypercyclic.

\end{proof}

 In particular, the  example of a weakly ergodic $3$-isometry   defined  in \cite[Section 5.2]{AS} is weak numerically hypercyclic.

\begin{question}
Do there exist  numerically hypercyclic weakly ergodic $3$-isometries?
\end{question}

Let $T$ be an $m$-isometry. What can we say about dynamical properties of $T^*$?
 Some particular classes of operators allow the study of the (chaotic) dynamics of the adjoints.

\begin{theorem}
Let  $S_w$ be  a  forward weighted shift strict $m$-isometry on $\ell^{2}(\mathbb{N})$. Then
\begin{enumerate}
\item $S_w^*$ is mixing if and only if $ m\ge 2$.

\item $S_w^*$ is chaotic if and only if $ m\ge 3$.

\end{enumerate}
\end{theorem}
\begin{proof}
By \cite[Theorem 1]{AL}, a unilateral  weighted forward shift on a Hilbert space is an $m$-isometry if and only if there exists a polynomial $p$ of degree at most $m-1$ such that for any integer $n\ge 1$, we have that  $p(n)>0$ and $|w_n|^2 =\displaystyle \frac{p(n+1)}{p(n)}$.
Thus for $m\geq 2$, $S_w^*$ satisfies condition ii) of (c) from   \cite[Theorem 4.8]{GEP11} and   $S_w^*$ is mixing. For $m\ge 3$, $S_w^*$ satisfies  condition ii) of c) from   \cite[Theorem 4.8]{GEP11} and    $S_w^*$ is chaotic.

\end{proof}

Notice that, if $S_w$ is a  unilateral forward weighted shift and a strict $m$-isometry on $\ell^{2}(\mathbb{N})$ with $m\ge 2$, then $S_w^*$ is hypercyclic operator.

Since on $\ell^{2}(\mathbb{Z})$ there exist bilateral forward weighted shifts which are strict $m$-isometries   only for  odd $m$, then we have
\begin{theorem}
Let  $S_w$ be   a  bilateral forward weighted shift strict $m$-isometry on $\ell^{2}(\mathbb{Z})$ with $m>1$. Then  $S_w^*$ is chaotic.
\end{theorem}

\begin{proof}
By \cite[Theorem 19 \& Corollary 20]{AL}, a bilateral weighted forward shift on a Hilbert space is a strict  $m$-isometry if and only if there exists a polynomial $p$ of degree at most $m-1$ such that for any integer $n$, we have $p(n)>0$ and $|w_n|^2 =\displaystyle \frac{p(n+1)}{p(n)}$ and $m$ is an odd integer.
Hence, for   $m\ge 3$,   $S_w^*$ satisfies condition  ii) of c) from  \cite[Theorem 4.13]{GEP11}. Thus  $S_w^*$ is chaotic.

\end{proof}

\small

\smallskip
\noindent T. Berm\'{u}dez\\
{\it Departamento de An\'alisis Matem\'atico, Universidad de la Laguna,
38271, La Laguna (Tenerife), Spain.}\\
{\it e-mail:} tbermude@ull.es

\smallskip
\noindent A. Bonilla\\
{\it Departamento de An\'alisis Matem\'atico, Universidad de la Laguna,
38271, La Laguna (Tenerife), Spain.}\\
{\it e-mail:} abonilla@ull.es

\smallskip
\noindent V. M\"uller\\
Mathematical Institute, Czech Academy of Sciences, Zitn\'a 25,
115 67 Prague 1, Czech Republic.\\
{\it e-mail}: muller@math.cas.cz

\smallskip
\noindent A. Peris\\
{\it IUMPA, Universitat Polit\`ecnica de Val\`encia, Departament de
Matem\`atica Aplicada, Edifici 7A, 46022 Val\`encia, Spain.}\\
{\it e-mail:} aperis@mat.upv.es

\end{document}